\documentclass{amsart}%
\usepackage{amsfonts}
\usepackage{amsmath}
\usepackage{amssymb}
\usepackage{graphicx}
\usepackage{xcolor}%
\providecommand{\U}[1]{\protect\rule{.1in}{.1in}}
\newtheorem{theorem}{Theorem}
\theoremstyle{plain}

\newtheorem{corollary}{Corollary}

\newtheorem{definition}{Definition}
\newtheorem{example}{Example}

\newtheorem{lemma}{Lemma}

\newtheorem{proposition}{Proposition}

\numberwithin{equation}{section}

\usepackage[multiple]{footmisc}
\usepackage{bigfoot}

\begin{document}
\title[On $S$-$1$-absorbing primary submodules]{On $S$-$1$-absorbing primary submodules}
\author{Mohammed Issoual}
\address{CRMEF Meknes, Morocco.} 
\email{issoual2@yahoo.fr}
\author{Najib Mahdou}
\address{Laboratory of Modeling and Mathematical Structures \\Department of Mathematics, Faculty of Science and Technology of Fez, Box 2202,
	University S.M. Ben Abdellah Fez, Morocco.}
\email{mahdou@hotmail.com}
\author{Neslihan Aysen Ozkirisci}
\address{Department of Mathematics, Yildiz Technical University, Davutpasa Campus, Esenler, 34210, Istanbul, Turkey}
\email{aozk@yildiz.edu.tr}
\author{Ece Yetkin Celikel}
\address{Department of Electrical-Electronics Engineering, Faculty of Engineering,
	Hasan Kalyoncu University, Gaziantep, Turkey}
\email{ece.celikel@hku.edu.tr, yetkinece@gmail.com}

\thanks{This paper is in final form and no version of it will be submitted for
publication elsewhere.}

\begin{abstract}
In this work, we introduce the notion of $S $-1-absorbing primary submodule as an extension of 1-absorbing primary submodule. Let $S $ be a multiplicatively closed subset of a ring $R $ and $M $ be an $R $-module. A submodule $N $ of $M $ with $(N:_{R}M)\cap S=\emptyset$ is said to be $S $-1-absorbing primary if whenever $abm\in N $ for some non-unit $a,b\in R $ and $m\in M $, then either $sab\in(N:_{R}M)$ or $sm\in M$-$rad(N)$. We examine several properties of this concept and provide some characterizations. In addition, $S $-1-absorbing primary avoidance theorem and $S $-1-absorbing primary property for idealization and amalgamation are presented.

\end{abstract}
\maketitle

\section{Introduction}\label{sec1}

In this article, we assume all rings are commutative with non-zero identity
and all modules are unital. We will start with certain definitions and
notations that will be used throughout the paper. Let $I$ be an ideal of a
ring $R$ and $K,L$ be two submodules of $R$-module $M.$ The residual $K$ by
$L$ and $I$ is defined as $(K:_{R}L)=\{r\in R\mid rL\subseteq K\}$ and
$(K:_{M}I)=\{m\in M\mid Im\subseteq K\},$ respectively. If $I$ is a principal
ideal of $R$ generated by an element $r\in R,$ then $(K:_{M}r)$ is preferred
instead of $(K:_{M}Rr).$ By $M$-$rad(N)$, we denote the $M$-radical of $N$
which is the intersection of all prime submodules containing the submodule
$N$. 

Consider a multiplicatively closed subset $S$ of a ring $R$. Recently,
$S$-versions and their features of prime and primary ideals have been
thoroughly researched in various studies. Moreover, absorbing-types of such
ideals have been proposed and investigated. From the literature, we have
gathered the following definitions that we are interested in:

\begin{enumerate}
	\item A submodule $N$ of $M$ with $(N:_{R}M)\cap S=\emptyset$ is said to be
	$S$-1-absorbing prime, if there exists a fixed element $s\in S$ whenever
	$abm\in N$ for some non-unit elements $a,b\in R$ and $m\in M$, then
	$sab\in(N:_{R}M)$ or $sm\in N$, \cite{Far}.
	
	\item A proper ideal $I$ of $R$ is called $1$-absorbing primary if whenever
	$abc\in I$ for some non-unit $a,b,c\in R$, then either $ab\in I$ or $c\in
	\sqrt{I}$, \cite{ece1}.
	
	\item A proper submodule $N$ of $M$ is called $1$-absorbing primary if
	whenever $abm\in N$ for some non-unit $a,b\in R$ and $m\in M$, then either
	$ab\in(N:_{R}M)$ or $m\in M$-$rad(N)$, \cite{Ece}.
	
	\item A submodule $N$ of $M$ with $(N:_{R}M)\cap S=\emptyset$ is said to be
	$S$-primary, if there exists a fixed element $s\in S$ such that whenever
	$am\in N$ for some $a\in R$ and $m\in M$, then $sa\in\sqrt{(N:_{R}M)}$ or
	$sm\in N$, \cite{Ans}.
	
	\item A submodule $N$ of $M$ with $(N:_{R}M)\cap S=\emptyset$ is said to be
	$S$-2-absorbing primary, if there exists a fixed element $s\in S$ such that
	whenever $abm\in N$ for some $a,b\in R$ and $m\in M$, then $sam\in N$ or
	$sbm\in N$ or $sab\in\sqrt{(N:_{R}M)}$, \cite{Naji}.
\end{enumerate}

Our purpose in this paper is to extend the notion of $S$-1-absorbing primary
ideals and introduce the concept of $S$-$1$-absorbing primary submodules. This
study is composed of five sections; in the second section, the concept of
$S$-$1$--absorbing primary submodule is defined and examined in detail.
Several characterizations of $S$-$1$--absorbing primary submodules are also
presented. We exemplify the differences between $S$-$1$--absorbing primary and
$S$-$1$--absorbing prime submodules, as well as $1$--absorbing primary
submodules. We provide some theorems to examine the relations between $S$%
-$1$--absorbing primary submodules and the other notions such as $S$-primary
and $1$--absorbing primary submodules.

In \cite{Lu}, it has been proposed and investigated the prime avoidance
theorem for modules. This theorem says that if a submodule $N $ of $M $ is
contained in a union of finitely many prime submodules $P_{i} $'s with
$(P_{j}:M)\nsubseteq(P_{k}:M) $ whenever $j\neq k $, then $N $ is contained in
$P_{i} $ for some $i $. This theorem has also been extended to primary
submodules in various ways \cite{AtTe, DS}. In the third section, we
provide $S$-$1$--absorbing primary avoidance theorem under the assumption that
$M $ is finitely generated multiplication.

In section 4, we examine the $S$-$1$--absorbing primary property for
idealization of $M $ in $R $. The idealization of $M $ in $R $ is the ring
$R(+)M = \{(r,m) \mid r\in R, m\in M\}$ with component-wise addition and
multiplication defined as $(r_{1},m_{1})(r_{2},m_{2})=(r_{1}r_{2},r_{1}%
m_{2}+r_{2}m_{1})$. For more information and the ideal structure of an
idealization of $M $, the reader may see \cite{ide}.

Let $(A,B)$ be a pair of commutative rings with unity, $f:A\rightarrow B$ be a
ring homomorphism and $J$ be an ideal of $B$. In this setting, we can consider
the following subring of $A\times B:$
\[
A\bowtie^{f}J:=\left\{ (a,f(a)+j)\mid a\in A, j\in J\right\}
\]
called the amalgamation of $A$ with $B$ along $J$ with respect to $f$. This
construction is defined and examined by D'Anna et al. \cite{DF} and then
extensively investigated in different papers \cite{MCM, MES, KMSN}. In the
last section, we study $S $-1-absorbing primary ideals of the amalgamation
$A\bowtie^{f}J.$

\section{Characterization of $S$-1-absorbing primary submodules}\label{sec2}

\begin{definition}
	Let $S$ be a multiplicatively closed subset of a ring $R$ and $M$ an
	$R$-module. We call a submodule $N$ of $M$ with $(N:_{R}M)\cap S=\emptyset$ an
	$S$-$1$--absorbing primary if there exists an (fixed) $s\in S$ such that for
	all non-unit $a,b\in R$ and $m\in M$ if $abm\in N$, then either $sab\in
	(N:_{R}M)$ or $sm\in M$-$rad(N)$. This fixed element $s\in S$ is called an
	$S$-element of $N.$
\end{definition}

In particular, if one take $M=R$ as an $R$-module, then we have the definition
of $S$-$1$--absorbing primary ideal of $R$ as following:

\begin{definition}
	Let $S$ be a multiplicatively closed subset of a ring $R$. We call a proper
	ideal $I$ of $R$ with $I\cap S=\emptyset$ an $S$-$1$--absorbing primary if
	there exists an (fixed) $s\in S$ such that whenever $abc\in I$ for some
	non-unit $a,b,c\in R$, then either $sab\in I$ or $sc\in\sqrt{I}$. This fixed
	element $s\in S$ is called an $S$-element of $I.$
\end{definition}

Let $N$ be a proper submodule of $M$. It is clear that any $S$-1-absorbing
prime submodule is a $S$-1-absorbing primary submodule and that two concepts
coincide if $M$-$rad(N)=N$. However, this implication is not reversible in
general, see the following example.

\begin{example}
	\label{e1}\textbf{(}$S$\textbf{-1-absorbing primary submodule that is not }%
	$S$\textbf{-1-absorbing prime)} Consider the $%
	%TCIMACRO{\U{2124} }%
	%BeginExpansion
	\mathbb{Z}
	%EndExpansion
	$-module $M=%
	%TCIMACRO{\U{2124} }%
	%BeginExpansion
	\mathbb{Z}
	%EndExpansion
	\times%
	%TCIMACRO{\U{2124} }%
	%BeginExpansion
	\mathbb{Z}
	%EndExpansion
	$ and the submodule $N=3%
	%TCIMACRO{\U{2124} }%
	%BeginExpansion
	\mathbb{Z}
	%EndExpansion
	\times8%
	%TCIMACRO{\U{2124} }%
	%BeginExpansion
	\mathbb{Z}
	%EndExpansion
	$ of $M$. For the multiplicatively closed subset $S=\{3^{n}:n>0\}$, $N$ is an
	$S$-1-absorbing primary submodule which is not $S$-1-absorbing prime.
\end{example}

\begin{proof}
	Suppose that $ab(m_{1},m_{2})\in N$ for some integers $m_{1},m_{2}$ and
	$a,b\neq\pm1$. If $m_{2}$ is even, then $sm\in3%
	%TCIMACRO{\U{2124} }%
	%BeginExpansion
	\mathbb{Z}
	%EndExpansion
	\times2%
	%TCIMACRO{\U{2124} }%
	%BeginExpansion
	\mathbb{Z}
	%EndExpansion
	=M$-$rad(N)$ for all $s\in S$. If $m_{2}$ is odd, then since $abm_{2}\in8%
	%TCIMACRO{\U{2124} }%
	%BeginExpansion
	\mathbb{Z}
	%EndExpansion
	$, we conclude that $ab\in8%
	%TCIMACRO{\U{2124} }%
	%BeginExpansion
	\mathbb{Z}
	%EndExpansion
	$, and so $sab\in24%
	%TCIMACRO{\U{2124} }%
	%BeginExpansion
	\mathbb{Z}
	%EndExpansion
	=(N:_{R}M)$ for all $s\in S.$ Thus $N$ is an $S$-1-absorbing primary
	submodule. As $2\cdot3\cdot(1,4)\in N$ but neither $6s\in(N:_{R}M)=24%
	%TCIMACRO{\U{2124} }%
	%BeginExpansion
	\mathbb{Z}
	%EndExpansion
	$ nor $s(1,4)\in N$ for any $s\in S$, $N$ is not $S$-1-absorbing prime.
\end{proof}

Let $S$ be a multiplicatively closed subset of a ring $R.$ In the following
proposition, by $S^{\ast},$ we denote the saturation set of $S$ defined by
$\{r\in R:\frac{r}{1}\in U(S^{-1}R)\}.$ It is clear that $S\subseteq S^{\ast}$
and $S^{\ast}$ is also a multiplicatively closed subset of $R$. In case of
$S=S^{\ast}$, then $S$ is said to be a saturated set, \cite{Gilmer}. For an
$R$-module $M,$ by $U_{M}(R)$, we denote the set $\{r\in R:rM=M\}$ which
contains $U(R)$ is a saturated multiplicatively closed subset of $R$.

\begin{proposition}
	\label{p1}Let $S$ be a multiplicatively closed subset of a ring $R$, and $N$
	be a submodule of an $R$-module $M$.
\end{proposition}

\begin{enumerate}
	\item Let $T\subseteq S$ be two multiplicatively closed subsets of $R.$ If $N$
	is an $T$-1-absorbing primary submodule such that $(N:_{R}M)\cap S=\emptyset$,
	then $N$ is also an $S$-1-absorbing primary submodule. Conversely, if for each
	$s\in S$, there is an element $s^{\prime}\in T$ such that $ss^{\prime}\in T$
	and $N$ is a $S$-1-absorbing primary submodule of $M$, then $N$ is an
	$T$-1-absorbing primary submodule of $M$.
	
	\item Every $1$-absorbing primary submodule $N$ with $(N:_{R}M)\cap
	S=\emptyset$, is $S$-1-absorbing primary. However, it is clear that these two
	classes of $1$-absorbing primary and $S$-1-absorbing primary submodules
	coincide if $S\subseteq U(R).$
	
	\item $N$ is an $S$-1-absorbing primary\ submodule of $M$ if and only if $N$
	is $S^{\star}$-1-absorbing primary.
	
	\item Suppose that $M=Rm$ is a cyclic $R$-module and $S\subseteq U_{M}(R).$
	Then $N$ is a 1-absorbing primary submodule if and only if $N$ is an
	$S$-1-absorbing primary submodule.
\end{enumerate}

\begin{proof}
	(1) It is clear.
	
	(2) The claim follows by taking $T=\{1\}$ in (1).
	
	(3) Suppose that $N$ is $S$-1-absorbing primary submodule. It is clear that
	$(N:_{R}M)\cap S^{\ast}=\emptyset$. Since $S\subseteq S^{\star}$, then by (1),
	$N$ is $S^{\star}$-1-absorbing primary. For the converse, let $abm\in N$ for
	some non-unit elements $a,b\in R$ and $m\in M$. Since $N$ is an $S^{\star}%
	$-1-absorbing primary submodule, there exists an $s\in S^{\star}$ such that
	$sab\in(N:_{R}M)$ or $sm\in M$-$rad(N)$. Since $\frac{s}{1}$ is a unit of
	$S^{-1}R$, $sur\in S$ for some $u\in S$ and $r\in R$. Say $t:=sur\in S.$
	Hence, we conclude $tm\in M$-$rad(N)$ or $tab\in(N:_{R}M)$, as needed.
	
	(4) Suppose that $N$ is an $S$-1-absorbing primary submodule. Let
	$abm^{\prime}\in N$ for some non-unit $a,b\in R$ and $m^{\prime}=rm\in M$.
	Then there exists an $s\in S$ such that $sab\in(N:_{R}M)$ or $srm\in
	M$-$rad(N).$ If $srm\in M$-$rad(N)$, then since $S\subseteq U_{M}(R)$ we have
	$srRm=rM\subseteq M$-$rad(N)$ which implies $m^{\prime}=rm\in M$-$rad(N)$. If
	$sab\in(N:_{R}M)$, then since $M$ is a multiplication module and $S\subseteq
	U_{M}(R),$ we conclude $abM=sabM\subseteq(N:_{R}M)M\subseteq M$-$rad(N).$ Thus,
	$N$ is a 1-absorbing primary submodule. The converse part follows from (2).
\end{proof}

\begin{example}
	\label{e2}\textbf{(}$S$\textbf{-1-absorbing primary submodule that is not
		1-absorbing primary)} Consider the multiplicatively closed subset
	$S=\{2^{k}3^{t}:$ $k,t\geq0\}\ $of $%
	%TCIMACRO{\U{2124} }%
	%BeginExpansion
	\mathbb{Z}
	%EndExpansion
	.$ Then the $%
	%TCIMACRO{\U{2124} }%
	%BeginExpansion
	\mathbb{Z}
	%EndExpansion
	$-submodule $N=(0)$ of $M=%
	%TCIMACRO{\U{211a} }%
	%BeginExpansion
	\mathbb{Q}
	%EndExpansion
	\times%
	%TCIMACRO{\U{2124} }%
	%BeginExpansion
	\mathbb{Z}
	%EndExpansion
	_{2}\times%
	%TCIMACRO{\U{2124} }%
	%BeginExpansion
	\mathbb{Z}
	%EndExpansion
	_{3}$ is an $S$-1-absorbing primary submodule which is not 1-absorbing primary.
\end{example}

\begin{proof}
	Suppose that $ab\left(  \frac{c}{d},\overline{x},\overline{y}\right)  \in N$
	for some $\pm1\neq a,b\in%
	%TCIMACRO{\U{2124} }%
	%BeginExpansion
	\mathbb{Z}
	%EndExpansion
	$ and $\left(  \frac{c}{d},\overline{x},\overline{y}\right)  \in M$. Consider
	the element $6\in S.$ Here, $abc=0$ which implies $a=0$ or $b=0$ or $c=0$. If
	$a=0$ or $b=0$, then clearly $sab\in(N:_{%
		%TCIMACRO{\U{2124} }%
		%BeginExpansion
		\mathbb{Z}
		%EndExpansion
	}M)$. If $c=0,$ then $6\left(  \frac{c}{d},\overline{x},\overline{y}\right)
	=(0,\overline{0},\overline{0})\in M$-$rad(N)=(0)$, and thus $N$ is an
	$S$-1-absorbing primary submodule. However, it is not $1$-absorbing primary as
	$2\cdot3\cdot(0,\overline{1},\overline{1})\in N$ but neither $2\cdot3\in(N:_{%
		%TCIMACRO{\U{2124} }%
		%BeginExpansion
		\mathbb{Z}
		%EndExpansion
	}M)=0$ nor $(0,\overline{1},\overline{1})\in M$-$rad(N)=(0)$.
\end{proof}

\begin{theorem}
	\label{char}Let $S$ be a multiplicatively closed subset of a ring $R.$ For a
	submodule $N$ of an $R$-module $M$ satisfying $(N:_{R}M)\cap S=\emptyset$, the
	following are equivalent.
\end{theorem}

\begin{enumerate}
	\item $N$ is an $S$-$1$-absorbing primary submodule of $M$.
	
	\item There exists an $s\in S$ such that whenever $Ibm$ $\subseteq N$ for some
	proper ideal $I$ of $R,$ $b\in R\backslash U(R)$ and $m\in M$ implies either
	$sIb\subseteq\left(  N:_{R}M\right)  $ or $sm\in M$-$rad(N)$.
	
	\item There exists an $s\in S$ such that whenever $IJm$ $\subseteq N$ for some
	proper ideals $I$, $J$ of $R$ and $m\in M$ implies either $sIJ\subseteq\left(
	N:_{R}M\right)  $ or $sm\in M$-$rad(N)$.
	
	\item There exists an $s\in S$ such that whenever $IJK$ $\subseteq N$ for some
	proper ideals $I$, $J$ of $R$ and a submodule $K$ of $M$ implies either
	$sIJ\subseteq\left(  N:_{R}M\right)  $ or $sK\subseteq M$-$rad(N)$.
\end{enumerate}

\begin{proof}
	(1)$\Rightarrow$(2). Let $s$ be an $S$-element of $N$. Suppose that $Ibm$
	$\subseteq N$ and $sIb\nsubseteq\left(  N:_{R}M\right)  $. Then $sxb\notin
	\left(  N:_{R}M\right)  $ for some non-unit $x\in I$. Since $xbm\in N$ and
	$sxb\notin\left(  N:_{R}M\right)  ,$ we have $sm$ $\in M$-$rad(N),$ as needed.
	
	(2)$\Rightarrow$(3). Let $s$ be the fixed element in (2). Suppose that $IJm$
	$\subseteq N$ but $sIJ\nsubseteq\left(  N:_{R}M\right)  $ for some proper
	ideals $I$, $J$ of $R$ and $m\in M$. Then $sIb\notin\left(  N:_{R}M\right)  $
	for some non-unit $b\in J$. From our assumption (2), we conclude $sm$ $\in
	M$-$rad(N)$.
	
	(3)$\Rightarrow$(4). Let $s$ be the fixed element in (3). Suppose $IJK$
	$\subseteq N$ but $sK$ $\nsubseteq N$ for some proper ideals $I$, $J$ of $R$
	and a submodule $K$ of $M.$ Hence $sm\notin M$-$rad(N)$ for some $m\in K$ which
	implies $sIJ\subseteq\left(  N:_{R}M\right)  $ by (3).
	
	(4)$\Rightarrow$(1). Let $a,b\in R$ be non-unit elements, $m\in M$ and $abm\in
	N$. Put $I$ $=aR$, $J=bR$ and $K=Rm$ in (4), then we are done.
\end{proof}

In the following, we characterize $S$-$1$-absorbing primary ideals of a ring
$R$ by taking $M=R$ as an $R$-module in Theorem \ref{char}.

\begin{theorem}
	\label{char1}Let $S$ be a multiplicatively closed subset of a ring $R$ and $I$
	an ideal of $R$ with $I\cap S=\emptyset$. Then $I$ is an $S$-$1$-absorbing
	primary ideal of $R$ if and only if there exists an $s\in S$ such that
	whenever $I_{1}I_{2}I_{3}\subseteq I$ for some proper ideals $I_{1},I_{2}$ and
	$I_{3}$ of $R$ implies either $sI_{1}I_{2}\subseteq I$ or $sI_{3}%
	\subseteq\sqrt{I.}$
\end{theorem}

Recall from \cite[Lemma 2.4]{Hoj} that if $M$ is a finitely generated
multiplication $R$-module, then $\sqrt{(N:_{R}M)}=(M$-$rad(N):_{R}M)$ for any
submodule $N$ of $M$

Next, we give a characterization for $S$-$1$-absorbing primary submodules in
terms of submodules of a finitely generated multiplication modules.

\begin{theorem}
	Let $S$ be a multiplicatively closed subset of a ring $R$ and $M$\ be a
	finitely generated multiplication $R$-module. A proper submodule $N$ of $M$
	with $(N:_{R}M)\cap S=\emptyset$ is an $S$-$1$-absorbing primary submodule of
	$M$ if and only if there exists an $s\in S$ such that whenever $N_{1}%
	N_{2}N_{3}\subseteq N$ for some proper submodules $N_{1}$, $N_{2},$ $N_{3}$ of
	$M$ implies either $sN_{1}N_{2}\subseteq N$ or $sN_{3}\subseteq M$-$rad(N)$.
\end{theorem}

\begin{proof}
	Suppose that $N_{1}N_{2}N_{3}$ $\subseteq N$ for some proper submodules
	$N_{1}=I_{1}M,$ $N_{2}=I_{2}M$ and $N_{3}=I_{3}M$ where $I_{1}$, $I_{2},$
	$I_{3}$ are proper ideals of $R$. Hence $I_{1}I_{2}I_{3}\subseteq(N:_{R}M)$
	and $(N:_{R}M)$ is a $S$-$1$-absorbing primary ideal of $R$ by Lemma \ref{d}.
	This yields that there exists $s\in S$ such that $sI_{1}I_{2}\subseteq\left(
	N:_{R}M\right)  $ or $sI_{3}\subseteq\sqrt{(N:_{R}M)}$ by Theorem \ref{char1}.
	Thus we conclude either $sN_{1}N_{2}=sIJM\subseteq\left(  N:_{R}M\right)  M=N$
	or $sN_{3}\subseteq\sqrt{(N:_{R}M)}M=M$-$rad(N)$, as desired. For the converse,
	suppose that $IJm$ $\subseteq N$ for some proper ideals $I$, $J$ of $R$ and
	$m\in M$. Put $N_{1}=IM$, $N_{2}=JM$ and $N_{3}=Rm$. This follows that
	$sIJ\subseteq(N:_{R}M)$ or $sm\in sN_{3}\subseteq M$-$rad(N)$, thus we are done
	by Theorem \ref{char}.
\end{proof}

Now, we are ready for a characterization of $S$-1-absorbing primary submodules
of finitely generated faithful multiplication modules in terms of presentation
ideals of them.

\begin{theorem}
	\label{char2}Let $M$ be a finitely generated faithful multiplication
	$R$-module and $S$ be a multiplicatively closed subset of $R.$ For a submodule
	$N$ of $M$ provided $(N:_{R}M)\cap S=\emptyset,$ the following statements are equivalent.
\end{theorem}

\begin{enumerate}
	\item $N$ is a $S$-1-absorbing primary submodule of $M$.
	
	\item $N=IM$ for some $S$-1-absorbing primary ideal $I$ of $R$.
	
	\item $(N:_{R}M)$ is an $S$-1-absorbing primary ideal of $R$.
\end{enumerate}

\begin{proof}
	(1)$\Rightarrow$(2)$.$ Let $N$ be a $S$-1-absorbing primary submodule of $M$.
	Then $(N:_{R}M)$ is an $S$-1-absorbing primary ideal of $R$ by Proposition
	\ref{d}. Since $M$ is multiplication, put $I=(N:_{R}M)$, so the claim is clear.
	
	(2)$\Rightarrow$(3)$.$ Since $M\ $is cancellation, that is,
	$IM=JM\Leftrightarrow I=J$ for every ideal $I,J\ $of $R$ which means that
	$(N:_{R}M)=(IM:_{R}M)=I$ is an $S$-1-absorbing primary ideal of $R$.
	
	(3)$\Rightarrow$(1)$.$ Suppose that $0\neq IJK\subseteq N\ $for some proper
	ideals $I,J\ $of $R\ $and a submodule $K\ $of $M.\ $Since $M$ is a
	multiplication module, $K=LM\ $for some ideal $L\ $of $R.\ $Then we have
	$IJLM\subseteq(N:_{R}M).\ $We may assume that $L\ $is a proper ideal of
	$R.\ $Since $(N:_{R}M)\ $is an $S$-1-absorbing primary ideal of $R,$ so it is
	an $S$-1-absorbing primary submodule of an $R$-module $R$ and there exists a
	fixed $s\in S$ such that $sIJ\subseteq(N:_{R}M)$ or $sL\subseteq\sqrt
	{(N:_{R}M)}=(M$-$rad(N):_{R}M)$ by Theorem \ref{char1}. This yields that
	$sIJ\subseteq(N:_{R}M)$ or $sK=LM\subseteq M$-$rad(N)$ and $N\ $is
	$S$-1-absorbing primary submodule of $M$ by Theorem \ref{char}.$\ $
\end{proof}

Every $S$-primary submodule of a finitely generated faithful multiplication module is $S$-1-absorbing primary, see the following.

\begin{proposition}
	Let $S$ be a multiplicatively closed subset of a ring $R$ and $M$ be a
	finitely generated faithful multiplication $R$-module. If $N$ is an
	$S$-primary submodule of $M,$ then it is $S$-1-absorbing primary.
\end{proposition}

\begin{proof}
	Let $N$ be an $S$-primary submodule of $M.$ Then $(N:_{R}M)$ is a $S$-primary
	ideal of $R$ by \cite[Proposition 2.11 (1)]{Ans} and hence $(N:_{R}M)$ is a
	$S$-1-absorbing primary ideal of $R.$ Thus $N$ is an $S$-1-absorbing primary
	submodule of $M$ by Theorem \ref{char2}.
\end{proof}

\begin{lemma}
	\label{lrad}Let $S$ be a multiplicatively closed subset of $R.$ If $I$ is an
	$S$-1-absorbing primary ideal, then $\sqrt{I}$ is an $S$-prime ideal of $R$.
\end{lemma}

\begin{proof}
	Let $s$ be an $S$-element of $I$, $a,b\in R$ and $ab\in\sqrt{I}$. If $a$ or
	$b$ is unit, then the claim is clear. So, assume that $a,b$ are non-unit
	elements. Then $a^{k}b^{k}=a^{t}a^{t}b^{k}\in I$ for some $k=2t>1$ which
	implies that $sa^{k}\in I$ or $sb^{k}\in\sqrt{I}$. Thus $sa\in\sqrt{I}$ or
	$sb\in\sqrt{I}$ and $\sqrt{I}$ is an $S$-prime ideal of $R$.
\end{proof}

\begin{proposition}
	Let $S$ be a multiplicatively closed subset of $R.$ If $N$ is an
	$S$-1-absorbing primary submodule of $M$, then $M$-$rad(N)$ is an $S$-prime
	submodule of $M$.
\end{proposition}

\begin{proof}
	Follows from Theorem \ref{char2}, Lemma \ref{lrad} and \cite[Proposition 2.9
	(2)]{Sev}.
\end{proof}

Let $S$ be a multiplicatively closed subset of a ring $R$. Then $S$ is said to
satisfy the maximal multiple condition if there exists $s\in S$ such that $t$
divides $s$ for each $t\in S$, \cite{And}. For instance, if $S$ is finite,
then $S$ satisfies the maximal multiple condition. Also, $U(R)$ is a
multiplicatively closed set satisfying the maximal multiple condition.

\begin{proposition}
	Let $S$ be a multiplicatively closed subset of a ring $R$, and $N$ be a
	submodule of an $R$-module $M$.
\end{proposition}

\begin{enumerate}
	\item If $N$ is an $S$-1-absorbing primary submodule of $M$, then $S^{-1}N$ is
	a 1-absorbing primary $S^{-1}R$-submodule of $S^{-1}M$.
	
	\item If $S$ satisfies the maximal multiple condition, and $S^{-1}N$ is a
	1-absorbing primary $S^{-1}R$-submodule of $S^{-1}M$, then $N$ is an
	$S$-1-absorbing primary submodule of $M.$
\end{enumerate}

\begin{proof}
	(1) Let $\frac{a}{s_{1}},\frac{b}{s_{2}}$ be non-unit elements of $S^{-1}R$
	and $\frac{m}{s_{3}}\in S^{-1}M$ with $\frac{a}{s_{1}}\frac{b}{s_{2}}\frac
	{m}{s_{3}}\in S^{-1}N$. Then $uabm\in N$ for some $u\in S$. Hence, there
	exists an $s\in S$ such that $suab\in(N:_{R}M)$ or $sm\in M$-$rad(N)$. This
	yields either $\frac{a}{s_{1}}\frac{b}{s_{2}}=\frac{suab}{sus_{1}s_{2}}\in
	S^{-1}(N:_{R}M)\subseteq(S^{-1}N:_{S^{-1}R}S^{-1}M)$ or $\frac{m}{s_{3}}%
	=\frac{sm}{ss_{3}}\in S^{-1}(M$-$rad(N))=S^{-1}M$-$rad(S^{-1}N)$ and $S^{-1}N$ is
	a 1-absorbing primary submodule of $S^{-1}M$.
	
	(2) Let $s_{\max}\in S$ is the element which has the maximal multiple
	condition. Suppose that $abm\in N$ for some non-unit elements $a,b\in R$ and
	$m\in M$. Then $\frac{a}{1}\frac{b}{1}\frac{m}{1}\in S^{-1}N$ which implies
	$\frac{a}{1}\frac{b}{1}\in(S^{-1}N:_{S^{-1}R}S^{-1}M)\subseteq$ $S^{-1}%
	(N:_{R}M)$ or $\frac{m}{1}\in S^{-1}M$-$rad(S^{-1}N)=S^{-1}(M$-$rad(N))$. This
	follows that $uab\in(N:_{R}M)$ or $vm\in M$-$rad(N)$. Since $u$ and $v$ divide
	$s_{\max}$, we get $s_{\max}ab\in(N:_{R}M)$ or $s_{\max}m\in M$-$rad(N)$ as needed.
\end{proof}

\begin{proposition}
	Let $S$ be a multiplicatively closed subset of $R,$ and $M$ be a
	multiplication $R$-module. If $\{N_{i}\}_{i=1}^{k}$ is a family of
	$S$-1-absorbing primary submodules of $M$ with the same $M$-radical, then $%
	%TCIMACRO{\dbigcap \limits_{i=1}^{k}}%
	%BeginExpansion
	{\displaystyle\bigcap\limits_{i=1}^{k}}
	%EndExpansion
	N_{i}$ is also $S$-1-absorbing primary submodule of $M.$
\end{proposition}

\begin{proof}
	Let $s_{i}\in S$ be an $S$-element of $N_{i}$ for each $i=1,...,k.$ First,
	note that $\left(
	%TCIMACRO{\dbigcap \limits_{i=1}^{k}}%
	%BeginExpansion
	{\displaystyle\bigcap\limits_{i=1}^{k}}
	%EndExpansion
	N_{i}\right)  \cap S=\emptyset$. Put $s=s_{1}s_{2}...s_{k}$. Suppose that
	$abm\in%
	%TCIMACRO{\dbigcap \limits_{i=1}^{k}}%
	%BeginExpansion
	{\displaystyle\bigcap\limits_{i=1}^{k}}
	%EndExpansion
	N_{i}$ but $sab\notin\left(
	%TCIMACRO{\dbigcap \limits_{i=1}^{k}}%
	%BeginExpansion
	{\displaystyle\bigcap\limits_{i=1}^{k}}
	%EndExpansion
	N_{i}:_{R}M\right)  $ for some non-unit elements $a,b\in R$ and $m\in M$. Then
	$sab\notin(N_{j}:_{R}M)$ for some $j\in\{1,\dots,k\}.$ Since $N_{j}$ is
	$S$-1-absorbing primary and $abm\in N_{j}$, we have $s_{j}m\in M$-$rad(N_{j})$
	and so $sm\in M$-$rad(N_{j})=M$-$rad\left(
	%TCIMACRO{\dbigcap \limits_{i=1}^{k}}%
	%BeginExpansion
	{\displaystyle\bigcap\limits_{i=1}^{k}}
	%EndExpansion
	N_{i}\right)  $ by \cite[Proposition 2.14 (3)]{Hoj}, we are done.
\end{proof}

\begin{lemma}
	\label{lemf}Let $f:M_{1}\rightarrow M_{2}$ be an $R$-module epimorphism and
	$S$ a multiplicatively closed subset of $R$.
\end{lemma}

\begin{enumerate}
	\item If $N$ is a submodule of $M_{1}$, then $(N:_{R}M_{1})\subseteq
	(f(N):_{R}M_{2}),$ \cite{Ece}.
	
	\item If $K$ is a submodule of $M_{2}$, then$(K:_{R}M_{2})\subseteq
	(f^{-1}(K):_{R}M_{1})$, \cite{Ece}.
	
	\item If $N$ is a submodule of $M_{1}$ and $\ker\left(  f\right)  \subseteq
	N$, then $f\left(  M_{1}\text{-}rad(N)\right)  =M_{2}$-$rad(f\left(  N\right)
	)$, \cite{Lu2}.
	
	\item If $K$ is a submodule of $M_{2}$, then $f^{-1}\left(  M_{2}%
	\text{-}rad(K)\right)  =M_{1}$-$rad(f^{-1}\left(  K\right)  )$, \cite{Lu2}.
\end{enumerate}

\begin{proposition}
	\label{f}Let $f:M_{1}\rightarrow M_{2}$ be an $R$-module epimorphism and $S$
	be a multiplicatively closed subset of $R$.
\end{proposition}

\label{p/}

\begin{enumerate}
	\item Suppose that $\left(  f\left(  N\right)  :_{R}M_{2}\right)  \cap
	S=\emptyset$. If $N$ is an $S$-1-absorbing primary submodule of $M_{1}$
	containing $Ker(f)$, then $f(N)$ is an $S$-1-absorbing primary submodule of
	$R_{2}.$
	
	\item If $K$ is an $S$-1-absorbing primary submodule of $M_{2}$, then
	$f^{-1}(K)$ is an $S$-1-absorbing primary submodule of $M_{1}.$
\end{enumerate}

\begin{proof}
	(1) Let $a,b$ be non-unit elements of $R$ and $m\in M$ with $abm\in f(N)$. Say
	$m=f(m_{1})$ for some $m_{1}\in M_{1}$. Since $abf(m_{1})\in f(N)$ and
	$Ker(f)\subseteq I$, we have $abm_{1}\in N$ which implies that there exists an
	$s\in S$ such that $sab\in(N:_{R}M_{1})$ or $sm\in M_{1}$-$rad(N)$. Thus,
	$sab\in(N:_{R}M_{1})\subseteq f((N):_{R}M_{2})$ or $sf(m)\in f(N)\subseteq
	M_{2}$-$rad(f(N))$ by Lemma \ref{lemf}.
	
	(2) First, we show that $\left(  f^{-1}\left(  K\right)  :_{R}M_{1}\right)
	\cap S=\emptyset.$ Assume that $r\in\left(  f^{-1}\left(  K\right)  :_{R}%
	M_{1}\right)  \cap S$. Then $rM_{1}\subseteq f^{-1}\left(  K\right)  $ implies
	$rM_{2}=rf(M_{1})\subseteq f(f^{-1}\left(  K\right)  )\subseteq K$, and hence
	$r\in(K:_{R}M_{2})$, a contradiction. Let $a$ and $b$ be non-unit elements of
	$R$ and $m\in M_{1}$ with $abm\in f^{-1}(K).$ Then $abf(m)\in K$ which implies
	that there exists $s\in S$ such that $sab\in(K:_{R}M_{2})$ or $sf(m)\in
	M_{2}$-$rad(K)$. Thus, we conclude $sab\in(K:_{R}M_{2})\subseteq(f^{-1}%
	(K):_{R}M_{1})$ or $sm\in f^{-1}\left(  M_{2}\text{-}rad(K)\right)
	=M_{1}$-$rad(f^{-1}(K))$ by Lemma \ref{lemf}. Therefore, $f^{-1}(K)$ is an
	$S$-1-absorbing primary submodule of $M_{1}.$
\end{proof}

As a consequence of Theorem \ref{f}, we have the following result.

\begin{corollary}
	\label{c/}Let $S$ be a multiplicatively closed subset of $R$ and $N,$ $K$ be
	submodules of $R$-module $M$.
\end{corollary}

\begin{enumerate}
	\item Let $K\subseteq N$. Then $N$ is an $S$-1-absorbing primary submodule of
	$M$ if and only if $N/K$ is an $S$-1-absorbing primary submodule of $M/K$.
	
	\item If $N$ is an $S$-1-absorbing primary submodule of $M$ with
	$(N:_{R}K)\cap S=\emptyset$, then $N\cap K$ is an $S$-1-absorbing primary
	submodule of $K$.
\end{enumerate}

\begin{proof}
	(1) It is clear that $\left(  N/K:_{R}M/K\right)  \cap S=\emptyset$. Consider
	the canonical epimorphism $\pi:M\rightarrow M/K$ in Proposition \ref{f} (1).
	Then $N/K$ is an $S$-1-absorbing primary submodule of $M/K$. For the converse
	part, suppose that $abm\in N$ for some non-unit elements $a$ and $b$ of $R$,
	$m\in M.$ Hence $ab(m+K)\in N/K$ which yields that there exists a fixed $s\in
	S$ such that $sab\in(N/K:_{R}M/K)$ or $s(m+K)\in M/K$-$rad(N/K)=M$-$rad(N)/K.$
	Therefore $sab\in(N:_{R}M)$ or $sm\in M$-$rad(N),$ so we are done.
	
	(2) Consider that the injection $i:K\rightarrow M$ defined by $i(k)=k$ for all
	$k\in K$. It is easy to see that $(i^{-1}(N):K)\cap S=\emptyset$. Then
	$i^{-1}(N)=N\cap K$ is is an $S$-1-absorbing primary submodule of $K$ by
	Proposition \ref{f} (2).
\end{proof}

Next, we show that if a ring $R$ admits an $S$-1-absorbing primary ideal that
is not an $S$-primary ideal, then $R$ is a quasilocal ring. First, recall the
following lemma.

\begin{lemma}
	\cite{ece1}\label{Lq} Let $R$ be a ring. Suppose that for every non-unit
	element $w$ of $R$ and for every unit element $u$ of $R$, we have $w+u$ is a
	unit element of $R$. Then $R$ is a quasilocal ring.
\end{lemma}

\begin{theorem}
	\label{Tq} Let $R$ be a ring. If there is an $S$-1-absorbing primary ideal
	which is not an $S$-primary, then $R$ is a quasilocal ring.
\end{theorem}

\begin{proof}
	Let $s^{\prime}\in S$ be an $S$-element of $J$ which is not $S$-primary ideal
	of $R$. Then there exist non-unit elements $a,b\in R$ with $ab\in J$
	satisfying $sa\notin J$ and $sy\notin\sqrt{J}$ for all $s\in S$. Choose $r$
	$\in R\backslash U(R)$. Hence $rab\in J$ which yields $s^{\prime}ra\in J$. We
	show that $r+u\in U(R)$ for any $u\in U(R)$. Let $u\in U(R)$. Assume that
	$r+u$ $\in R\backslash U(R)$. Then $(r+u)ab\in J$ and $s^{\prime}y\notin
	\sqrt{I}$ implies that $s^{\prime}(r+u)a\in I$. Now, since $s^{\prime}ra\in
	I,$ we get $s^{\prime}ua\in I$, and so $s^{\prime}a\in J$, a contradiction.
	Therefore, $r+u$ $\in U(R)$ and $R$ is a quasilocal ring by Lemma \ref{Lq}.
\end{proof}

As a direct result of Theorem \ref{Tq}, we have the following corollary.

\begin{corollary}
	\label{cq} Let $I$ be an ideal of a non-quasi local ring $R$ with $I\cap
	S=\emptyset$. Then $I$ is an $S$-1-absorbing primary ideal of $R$ if and only
	if $I$ is an $S$-primary ideal of $R$.
\end{corollary}

\begin{lemma}
	\label{d}Let $S$ be a multiplicatively closed subset of $R$. If $N$ is an
	$S$-1-absorbing primary submodule of a finitely generated multiplication
	$R$-module $M$. Then the following statements hold.
\end{lemma}

\begin{enumerate}
	\item $(N:_{R}M)$ is an $S$-1-absorbing primary ideal of $R$.
	
	\item If $(N:_{R}m)\cap S=\emptyset$ for some $m\in M\backslash\mathfrak{P}$
	where $\mathfrak{P}$ is any prime submodule of $M$ containing $N,$ then
	$(N:_{R}m)$ is an $S$-1-absorbing primary ideal of $R$.
	
	\item If $((N:_{M}r):_{R}M)\cap S=\emptyset$ for some $r\in R\backslash
	(N:_{R}M)$, then $(N:_{M}r)$ is an $S$-1-absorbing primary submodule of $M$.
\end{enumerate}

\begin{proof}
	(1) Let $a,b,c$ $\in R\backslash U(R)$ and $abc\in(N:_{R}M)$. Then
	$ab(cM)\subseteq N$ implies that there exists an $s\in S$ such that
	$sab\in(N:_{R}M)$ or $scM\subseteq M$-$rad(N)$ by Theorem \ref{char}. Thus we
	have $sab\in(N:_{R}M)$ or $sc\in (M$-$rad(N):_{R}M)  =\sqrt
	{(N:_{R}M)}$ and $(N:_{R}M)$ is an $S$-1-absorbing primary ideal of $R$.
	
	{(2) Let $a,b,c$ $\in R\setminus U(R)$ and $abc\in(N:_{R}m)$. Then we have
		$abcm\in N.$ As $N$ is an $S$-1-absorbing primary submodule, there exists an
		$s\in S$ such that $sab\in(N:_{R}M)$ or $scm\in M$-$rad(N).$ This indicates that
		$sab\in(N:_{R}m)$ or $sc\in(M$-$rad(N):_{R}m).$ For the rest of the proof, we
		only need to show that $(M$-$rad(N):_{R}m)=\sqrt{(N:_{R}m)}.$ The inclusion
		$\sqrt{(N:_{R}m)}\subseteq(M$-$rad(N):_{R}m)$ is satisfied by (\cite{MRS}, Lemma
		2.1). For the reverse inclusion, let $a\in(M$-$rad(N):_{R}m)$ which implies
		$am\in M$-$rad(N)=\bigcap_{N\subseteq P}P$ where $P$ is any prime submodule of
		$M$ that contains $N.$ Since $M$ is multiplication, there exists an ideal $I$
		of $R$ such that $N=IM$ and a prime ideal $Q$ of $R$ such that $P=QM.$ It
		follows that $am\in\bigcap_{N\subseteq P}P=\bigcap_{IM\subseteq QM}QM.$ Then,
		we obtain $am\in QM$ for all prime ideals $Q$ containing $I$ by using
		(\cite{El}, Theorem 3.2). This means that $a\in Q$ or $m\in QM$ for all primes
		$Q$ containing $I$ by (\cite{El}, Lemma 2.10) and hence $a\in\bigcap
		_{I\subseteq Q}Q=\sqrt{I}$ which implies $a^{n}\in I$ for some $n\in
		\mathbb{N}.$ We conclude that $a^{n}m\in IM=N$, namely $a\in\sqrt{(N:_{R}m)},$
		which completes the proof. }
	
	(3) Suppose that $abm\in(N:_{M}r)$ for some non-unit $a,b\in R$ and $m\in M$.
	Then $abrm\in N$ which implies that there exists an $s\in S$ such that
	$sab\in\left(  N:_{R}M\right)  \subseteq((N:_{M}r):_{R}M)$ or $srm\in
	M$-$rad(N)\subseteq M$-$rad((N:_{M}r))$. Thus $(N:_{M}r)$ is an $S$-1-absorbing
	primary submodule of $M$.
\end{proof}

\begin{theorem}
	Let $M$ be a finitely generated multiplication $R$-module where $R$ is a
	non-quasi local ring and $N$ be a submodule of $M$ with $(N:_{R}M)\cap
	S=\emptyset$. Then $N$ is an $S$-1-absorbing primary submodule of $M$ if and
	only if $N$ is an $S$-primary submodule of $M.$
\end{theorem}

\begin{proof}
	Suppose that $N$ is a $S$-1-absorbing primary submodule of $M$. Then
	$(N:_{R}M)$ is an $S$-1-absorbing primary ideal of $R$ by Lemma \ref{d} and
	$(N:_{R}M)$ is $S$-primary ideal of $R$ by Corollary \ref{cq}. Hence, $N$ is
	an $S$-primary submodule of $M$ by \cite[Theorem 2.17]{Ans}. Conversely, if
	$N$ is an $S$-primary submodule of $M,$ then $(N:_{R}M)$ is an $S$-primary
	ideal of $R$ by \cite[Theorem 2.17]{Ans}$.$ Since every $S$-primary ideal with
	$(N:_{R}M)\cap S=\emptyset$ is an $S$-1-absorbing primary ideal, $(N:_{R}M)$
	is an $S$-1-absorbing primary ideal of $R.$ Thus, $N$ is an $S$-1-absorbing
	primary submodule of $M$ by Theorem \ref{char}.
\end{proof}

In view of Corollary \ref{cq}, we have the following result.

\begin{theorem}
	Let $S_{i}$ be multiplicatively closed subset of ring $R_{i}$ for each
	$i=1,\dots,n$ and $S=S_{1}\times\cdots\times S_{n}$. For an ideal
	$I=I_{1}\times\cdots\times I_{n}$ of $R=R_{1}\times\cdots\times R_{n},$ the
	following are equivalent.
	
	\begin{enumerate}
		\item $I$ is an $S$-1-absorbing primary ideal of $R$.
		
		\item $I$ is an $S$-primary ideal of $R$.
		
		\item $I_{j}$ is an $S_{j}$-primary ideal of $R_{j}$ for some $j\in
		\{1,\dots,n\}$ and $I_{k}\cap S_{k}\neq\emptyset$ for all $k\in\{1,\dots
		,n\}\backslash\{j\}$.
	\end{enumerate}
	
	\begin{proof}
		(1)$\Leftrightarrow$(2) Since $R$ is not quasi-local, $S$-1-absorbing primary
		ideals and $S$-primary ideals coincide by Corollary \ref{cq}.
		(2)$\Leftrightarrow$(3) Follows from \cite[Theorem 2.21]{Ans} considering
		$M_{i}=R_{i}$ as an $R_{i}$-module for all $i\in\{1,\dots,n\}.$
	\end{proof}
\end{theorem}

\begin{proposition}
	Let $S$ be a multiplicatively closed subset of $R,$ $M_{1}$, $M_{2}$ be
	$R$-modules and $M=M_{1}\times M$ and $N_{1}$ be a proper submodule of $M_{1}%
	$. If $N=N_{1}\times M_{2}$ is an $S\times0$-1-absorbing primary submodule of
	$R$-module $M$, then $N_{1}$ is an $S$-1-absorbing primary submodule of
	$M_{1}.$
\end{proposition}

\begin{proof}
	Suppose that $N=N_{1}\times M_{2}$ is an $S\times0$-1-absorbing primary
	submodule of $M$. Then $N_{1}\cong N/(\{0\}\times M_{2})$ is an $S$%
	-1-absorbing primary submodule of $M/(\{0\}\times M_{2})\cong M_{1}$ by
	Corollary \ref{c/}.
\end{proof}

\begin{theorem}
	Let $S$ be a multiplicatively close subset of $R,$ $M_{1},$ $M_{2}$ be
	$R$-modules and $M=M_{1}\times M_{2}$ and $N_{1},N_{2}$ be proper submodules
	respectively of $M_{1}$ and $M_{2}.$ Let $M=M_{1}\times M_{2}$,
	$N=N_{1}\times M_{2}$ and $ N^{\prime}=M_{1}\times N_{2}.$ We suppose that
	$(N_{i}:M_{i})\cap S=\emptyset,$ where $i=1,2.$
	
	\begin{enumerate}
		\item $N_{1}$ is an $S$-1-absorbing primary submodule of $M_{1}$ if and only if
		$N$ is an $S\times\{1\}$-1-absorbing primary submodule of $M.$
		
		\item $N_{2}$ is an $S$-1-absorbing primary submodule of $M_{2}$ if and only if
		$N^{\prime}$ is a $\{1\}\times S$-1-absorbing primary submodule of $M.$
	\end{enumerate}
\end{theorem}

\begin{proof}
	
	(1) Suppose that $N$ is an $S$-1-absorbing primary submodule of $M,$ let
	$(s,1)$ be $S\times\{1\}$-element of $N.$ Let $a,b$ be non-nunit elements of
	$R$ and $m_{1}\in M_{1}$ with $abm_{1}\in N_{1},$ so $(a,1)(b,1)(m_{1}%
	,0)=(abm_{1},0)\in N.$ Hence $(s,1)(a,1)(b,1)\in(N_{1}\times M_{2}:M_{1}\times
	M_{2})$ or $(s,1)(m_{1},0)\in M$-$rad(N_{1}\times M_{2}).$ If
	$(s,1)(a,1)(b,1)\in(N_{1}\times M_{2}:M_{1}\times M_{2}),$ it is clear that
	$sab\in(N_{1}:M_{1}).$ Now suppose $(s,1)(m_{1},0)\in M$-$rad(N_{1}\times
	M_{2}).$ By \cite[Lemma 2.10]{Ebrahimi}, we have $(s,1)(m_{1},0)\in
	M_{1}$-$rad(N_{1})\times M_{2},$ and thus $sm_{1}\in M_{1}$-$rad(N_{1}).$ We conclude that $N_{1}$ is an $S$-1-absorbing primary submodule. Now assume that $N_{1}$ is an $S$-1-absorbing primary submodule of $M_{1}$ and let $s\in S$ be an $S$-element of $N_{1}.$ Let $(a,a^{\prime}), (b,b^{\prime})$ be non-units of $R\times R$
	and $(m,m^{\prime})\in M$ with $(a,a^{\prime})(b,b^{\prime})(m,m^{\prime})\in
	N.$ Then $abm\in N_{1}$ and so either $sab\in(N_{1}:M_{1})$ or $sm_{1}\in
	M_{1}$-$rad(N).$ If $sab\in N_{1},$ we conclude that $(s,1)(a,a^{\prime
	})(b,b^{\prime})=(sab,a^{\prime}b^{\prime})\in(N_{1}\times M_{2}:M_{1}\times
	M_{2}),$ now if $sab\in M_{1}$-$rad(N_{1}),$ we get $(s,1)(a,a^{\prime
	})(m,m^{\prime})=(sam,sa^{\prime}m^{\prime})\in M_{1}$-$rad(N_{1})\times
	M_{2}=M$-$rad(N_{1}\times M_{2})$ by \cite[Lemma 2.10]{Ebrahimi}. Hence, $N$ is an
	$S\times\{1\}$-1-absorbing primary submodule of $M.$
	
	(2) The similar proof of (1).
	
\end{proof}

\begin{lemma}
	\label{mrad}  Let $N $ be a submodule of $M $ and $S $ be a multiplicatively
	closed subset of $R. $ If $(P:M) \cap S= \emptyset$ for all prime submodule $P
	$ containing $N, $ then $(M$-$rad(N):s)=M$-$rad(N)=M$-$rad(N:s) $ for any $s\in S. $
\end{lemma}

\begin{proof}
	The inclusions $M$-$rad(N)\subseteq M$-$rad(N:s) $ and $M$-$rad(N)\subseteq
	(M$-$rad(N):s) $ are clear. Let $m\in(N:s) $ and suppose that $P $ is a prime
	submodule of $M $ containing $N. $ Then $sm\in N\subseteq P $ implies directly
	$m\in P $ due to our assumption. Thus, $(N:s)\subseteq P $ and we conclude
	$M$-$rad(N)=M$-$rad(N:s). $ Now, let $m\in(M$-$rad(N):s). $ Then $sm\in M$-$rad(N) $
	which means that $sm\in P $ for all prime submodules $P $ containing $N. $
	Again we have $m\in P $ by our assumption. It follows that $m\in M$-$rad(N) $
	and consequently we get $(M$-$rad(N):s)\subseteq M$-$rad(N). $
\end{proof}

In the following proposition, we denote $Z_{I}(R) $ by the set $\{x\in R \mid
xy\in I \text{ for some } y\in R\setminus I\}. $

\begin{proposition}
	\label{Ns}  Let $N $ be a submodule of $M $ and $S $ be a multiplicatively
	closed subset of $R $ provided $(P:M) \cap S= \emptyset$ for all prime
	submodule $P $ containing $N. $ If $Z_{(N:M)}(R) \cap S= \emptyset, $ then
	$(N:s) $ is a 1-absorbing primary submodule of $M $ for all $s\in S. $
\end{proposition}

\begin{proof}
	Suppose that $abm\in(N:s) $ for some non-unit $a,b\in R $ and $m\in M. $
	Hence, $sabm\in M $ and there exists $s^{\prime}\in S $ such that either
	$s^{\prime}sab\in(N:M) $ or $s^{\prime}m\in M$-$rad(N). $ If $s^{\prime}%
	sab\in(N:M), $ then we conclude $ab\in(N:M)\subseteq((N:s):M) $ since
	$Z_{(N:M)}(R) \cap S= \emptyset. $ Now, assume that $s^{\prime}m\in M$-$rad(N).
	$ Then we have $m\in(M$-$rad(N):s^{\prime})=M$-$rad(N)=M$-$rad(N:s) $ by Lemma
	\ref{mrad}, so we are done.
\end{proof}

\section{$S$-1-absorbing Primary Avoidance Theorem}\label{sec3}

In this section, $S $-1-absorbing primary avoidance theorem is proved.
Throughout this section, let $M $ be a finitely generated multiplication $R
$-module, $N, N_{1},N_{2}, \dots, N_{n} $ be submodules of $M $ and $S $ be a
multiplicatively closed subset of $R. $ Recall that a covering $N\subseteq
N_{1} \cup N_{2} \cup\dots\cup N_{n}$ is said to be efficient if no $N_{k} $
is unnecessary. Besides, if none of the $N_{k} $ may be excluded, the union
$N= N_{1} \cup N_{2} \cup\dots\cup N_{n}$ is efficient \cite{Lu}. We first
state some theorems that we will require.

\begin{lemma}
	(\cite{HM}, Corollary 1)\label{Cor1}  Let $R $ be a ring, $S \subseteq R $ a
	multiplicatively closed set and $P $ an ideal of $R $ disjoint with $S $. Then
	$P $ is $S $-prime if and only if there exists $s \in S $, such that for all
	$I_{1}, \dots, I_{n} $ ideals of R, if $I_{1} \dots I_{n} \subseteq P $, then
	$sI_{j} \in P $ for some $j\in\{1,\dots, n\}. $
\end{lemma}

\begin{lemma}
	(\cite{HM}, Proposition 4)\label{Prop4}  Let $R $ be a ring, $S \subseteq R $
	a multiplicatively closed set and $P $ an ideal of $R $ disjoint with $S $.
	Then $P $ is $S $-prime if and only if there exists $s \in S $ such that for
	all $x_{1}, \dots, x_{n} \in R $, if $x_{1} \dots x_{n} \in P $, then $sx_{j}
	\in P $ for some $j\in\{1,\dots, n\}. $
\end{lemma}

\begin{theorem}
	\label{A1} Let $N\subseteq N_{1}\cup N_{2}\cup\dots\cup N_{n}$ be an efficient
	covering of submodules $N_{1},N_{2},\dots,N_{n}$ of $M$ where $n>2$. Assume
	that $\sqrt{(N_{k}:m)}$ is an ideal of $R$ disjoint with $S$ {\ for all $1\leq
		k\leq n$ and all $m\in M\setminus P$ where $P$ is any prime submodule of $M$
		containing $N_{k}.$} If $\sqrt{(N_{j}:M)}\nsubseteq(\sqrt{(N_{k}:m)}:s)$ for
	all $s\in S$ whenever $j\neq k,$ then no $N_{k}$ is an $S$-1-absorbing primary
	submodule of $M.$
\end{theorem}

\begin{proof}
	It can be easily seen that $N=(N\cap N_{1})\bigcup(N\cap N_{2})\bigcup
	\dots\bigcup(N\cap N_{n})$ is an efficient union since $N\subseteq N_{1} \cup
	N_{2} \cup\dots\cup N_{n}$ is efficient.  Thus, there exists an element $m_{k}
	\in N \setminus N_{k}$ for all $k\leq n. $ From (\cite{Lu}, Lemma 2.2), it is
	known that $\bigcap_{j\neq k}(N\cap N_{j})\subseteq N \cap N_{k}. $ Now assume
	that $N_{k} $ is an $S $-1-absorbing primary submodule of $M $ for some
	$k\in\{1,\dots, n\}. $ Then, $(N_{k}:m) $ is an $S $-1-absorbing primary ideal
	of $R $ by Lemma \ref{d}. Furthermore, by Lemma \ref{lrad}, $\sqrt{(N_{k}:m)}
	$ is an $S $-prime ideal of $R. $ By our hypothesis, $s\sqrt{(N_{j}:M)}
	\nsubseteq(\sqrt{(N_{k}:m)} $ and then there exists a non-unit element
	$a_{j}\in\sqrt{(N_{j}:M)} $ such that $sa_{j}\notin\sqrt{(N_{k}:m)} $ for all
	$s\in S $ and all $m\in M \setminus N_{k}. $ (If $a_{j} $ is unit, then
	$\sqrt{(N_{j}:M)}=R $ and hence $N_{j}=M $ which contradicts with our
	assumption that the covering is efficient.) So, there exists a positive
	integer $n_{j} $ such that $a_{j}^{n_{j}}\in(N_{j}:M) $ where $j\neq k. $ From
	Lemma \ref{Cor1}, we also have $\prod_{j\neq k} \sqrt{(N_{j}:M)}
	\nsubseteq\sqrt{(N_{k}:m)} $ due to $\sqrt{(N_{k}:m)} $ is $S $-prime. Now let
	us say $a= \prod_{j\neq k}a_{j}$ and $t=max\{n_{j}\}_{j\neq k}. $ Then we get
	$a^{t} \in(N_{j}:M) $ and it follows that $a^{t}m_{k} \in N \cap N_{j} $ for
	every $j\neq k. $ However, $a \notin\sqrt{(N_{k}:m)} $ by Lemma \ref{Prop4}
	and this implies that $a^{t}m\notin N_{k} $ for all $m\in M \setminus N_{k}. $
	As a result, $a^{t}m_{k} \notin N_{k}$ because $m_{k} \notin N_{k}$ as well.
	This shows that $a^{t}m_{k} \in\bigcap_{j\neq k}(N\cap N_{j})\setminus N \cap
	N_{k}, $ which is a contradiction.
\end{proof}

\begin{theorem}
	[$S $-1-absorbing primary avoidance theorem] Let $N, N_{1},N_{2}, \dots, N_{n}
	$ be submodules of $M $ such that $N\subseteq N_{1} \cup N_{2} \cup\dots\cup
	N_{n}$ where at most two of $N_{1},N_{2}, \dots, N_{n} $ are not $S
	$-1-absorbing primary submodules. Suppose that $\sqrt{(N_{k}:m)} \cap
	S=\emptyset$ for all $m\in M \setminus N_{k} $ where $1\leq k\leq n $. If
	$\sqrt{(N_{j}:M)} \nsubseteq(\sqrt{(N_{k}:m)}:s) $ for all $s\in S $ whenever
	$j\neq k, $ then $N \subseteq N_{k} $ for some $k. $
\end{theorem}

\begin{proof}
	We can assume that the covering $N\subseteq N_{1} \cup N_{2} \cup\dots\cup
	N_{n}$ is efficient because any cover including submodules of $M $ can be
	reduced to an efficient covering by eliminating any superfluous terms and
	therefore $n \neq2 $. Since there is no $S $-1-absorbing primary submodule
	$N_{k} $ from Theorem \ref{A1}, $n<2 $. Hence we obtain $N \subseteq N_{k} $
	for some $k \in\{1,\dots, n\}. $
\end{proof}

\begin{corollary}
	Let $N $ be a submodule of $M. $ If the $S $-1-absorbing primary avoidance
	theorem holds for $M $, then the theorem holds also for $M/N. $
\end{corollary}

\begin{proof}
	Let $K/N, N_{1}/N, N_{2}/N, ..., N_{n}/N $ be submodules of $M/N $ such that
	at most two of $N_{1}/N, N_{2}/N, ..., N_{n}/N $ are not $S $-1-absorbing
	primary and $K/N \subseteq N_{1}/N \cap N_{2}/N \cap\dots\cap N_{n}/N $ where
	$n\geq2 $ . Thus, $K \subseteq N_{1} \cap N_{2} \cap\dots\cap N_{n} $ and at
	most two of $N_{1}, N_{2}, ..., N_{n} $ are not $S $-1-absorbing primary by
	Corollary 1. Assume that $\sqrt{N_{j}/N:M/N}\nsubseteq(\sqrt{N_{k}/N:m+N}:s)$
	for all $s\in S $ and all $m+N \in M/N\setminus N_{k}/N $ whenever $j\neq k. $
	Then $\sqrt{(N_{j}:M)} \nsubseteq(\sqrt{(N_{k}:m)}:s) $ for all $s\in S $.
	Since the $S $-1-absorbing primary avoidance theorem is satisfied by the
	submodules of $M, $ the rest of the proof is immediate.
\end{proof}

\section{Idealization of $S$-1-absorbing primary submodules}\label{sec4}

Let $M$ be a unitary $R$-module. Recall from \cite{Nagata} the idealization
$R(+)M$ of $M$ is the direct sum $R(+)M=R\oplus M$ which is a commutative ring
with componentwise addition and multiplication $(r_{1},m_{1})(r_{2},m_{2}%
)=(r_{1}r_{2},r_{1}m_{2}+r_{2}m_{1}).$  It is well known that if $N$ is a
submodule of $M$ with $IM\subset N,$ then $I(+)N$ is an ideal of $R(+)M.$

\begin{proposition}
	\label{id}Let $S$ be a multiplicatively closed subset of $R$, $I$ be an ideal
	of $R$ with $ I\cap S=\emptyset.$ Then the following statements are equivalent:
	
	\begin{enumerate}
		\item $I$ is an $S$-1-absorbing primary ideal of $R$.
		
		\item $I(+)M$ is an $S(+)0$-1-absorbing primary ideal of $R(+)M$.
		
		\item $I(+)M$ is an $S(+)M$-1-absorbing primary ideal of $R(+)M.$
	\end{enumerate}
\end{proposition}

\begin{proof}
	(1)$\Rightarrow$ (2) It is clear that $(I(+)N) \cap(S(+)0)=\emptyset$ if and
	only if $S\cap I=\emptyset.$ Suppose that $I$ is an $S$-1-absorbing primary ideal
	of $R$ and $s$ be an $S$-element of $I.$ Let $(a,m)(b,n)(c,p)\in I(+)M$ where
	$(a,m),(b,n),(c,p)$ are non-unit elements of $R(+)M.$ Since $U(R(+)M)=U(R)(+)M,$ we
	conclude that $a,b,c$ are non-units in $R$ and $abc\in I.$ Then either $sab\in
	I$ or $sc\in\sqrt{I}.$ If $sab\in I,$ then $(s,0)(a,m)(b,n)=(sab,san+sbm)\in
	I(+)M.$ If $sc\in\sqrt{I},$ then $(s,0)(c,p)=(sc,0)\in\sqrt{I(+)M}$ due to $\sqrt{I(+)M}=\sqrt{I}(+)M$. Thus $I(+)M$ is an $S(+)0$-1-absorbing
	primary ideal of $R(+)M.$
	
	(2)$\Rightarrow$ (3) Is clear since $S(+)0\subset S(+)M.$
	
	(3)$\Rightarrow$ (1) Assume that $I(+)M$ is an $S(+)M$-1-absorbing primary
	ideal of $R(+)M$. Let $(s,m)$ be an $S(+)M$-element of $I(+)M$ and let $abc\in
	I$ for non-unit elements $ a,b,c $ of $R.$ Then $(a,0),(b,0),(c,0)$ are also non-unit elements of $R(+)M$ since $U(R(+)M)=U(R)(+)M$. Besides, we have $(a,0)(b,0)(c,0)\in I(+)M.$ So, either $(s,m)(a,0)(b,0)=(sab,abm)\in I(+)M$ or $(s,m)(c,0)\in \sqrt{I(+)M}$ which yields $sab\in I$ or $sc\sqrt{I}.$ Hence $I$ is an $S(+)M$-1-absorbing primary ideal of $R.$
\end{proof}

%Let $M$ be an $R$-module, $N$ be a submodule of $M$ and $I$ be an ideal of $R$
%such that $IM\subset N,$ it is well known that $I(+)N$ is an ideal of $R(+)M.$

\begin{proposition}
	\label{id}Let $S$ be a multiplicatively closed subset of $R$ and $I$ be an ideal of $R$ with $I\cap S=\emptyset.$ Suppose $N$ is a proper submodule of $M$ such that $IM\subset N.$ Then $(1)\Longrightarrow (2)\Longrightarrow(3)$ is satisfied for the below statements:
	
	\begin{enumerate}
		\item $I(+)N$ is an $S(+)0$-1-absorbing primary ideal of $R(+)M$.
		
		\item $I(+)N$ is an $S(+)M$-1-absorbing primary ideal of $R(+)M.$
		
		\item $I$ is an $S$-1-absorbing primary ideal of $R$.
	\end{enumerate}
\end{proposition}

\begin{proof}
	$(1)\Longrightarrow(2)$ It is clear since $S(+)0\subset S(+)M.$
	
	$(2)\Longrightarrow(3)$ Let $abc\in I$ for non-unit elements $a,b,c\in R.$ We
	have $(a,0)(b,0)(c,0)\in I(+)N$ which is $S(+)M$-1-absorbing primary ideal of
	$R(+)M$ where $(a,0),(b,0),(c,0)$ are non-units in $R(+)M.$ Then either
	\[
	(s,m)(a,0)(b,0)=(sab,abm)\in I(+)N
	\]
	or%
	
	\[
	(s,m)(c,0)=(sv,cm)\in\sqrt{I(+)N}=\sqrt{I}(+)M.
	\]
	where $(s,m)$ is an $S(+)M$-element of $I(+)N.$ Then either $sab\in I$ or
	$sc\in\sqrt{I}.$ Hence $I$ is an $S$-1-absorbing primary ideal of $R.$
\end{proof}

\begin{example}
	The implication $(3)\Longrightarrow(1)$ is not true in general case. If $I$ is an
	$S$-1-absorbing primary ideal, then $I(+)N$ (where $N$ a proper submodule of $M$)
	is not necessarily $S(+)0$-1-absorbing primary ideal of $R(+)M.$ Take
	$R=\mathbf{Z}=M$ and $S=\{\mp1\}.$ Then $I=\{0\}$ is an $S$-1-absorbing
	primary ideal of $\mathbf{Z},$ since $(0(+)6\mathbf{Z}:(0,1))=6\mathbf{Z}(+)\mathbf{Z}$
	is not an $S(+)0$-prime ideal. We conclude that
	$0(+)6\mathbf{Z}$ is not an $S(+)0$-1-absorbing primary ideal of $\mathbf{Z}%
	(+)\mathbf{Z}$ by Lemma \ref{lrad}.
\end{example}

\section{Amalgamation of $S$-1-absorbing primary submodules}\label{sec5}

Let $A$ and $B$ be commutative rings with unity, let $J$ be an ideal of $B$
and let $f:A\rightarrow B$ be a ring homomorphism. In this setting, we can
consider the following subring of $A\times B:$
\[
A\bowtie^{f}J:=\left\{ (a,f(a)+j)\mid a\in A, j\in J\right\}
\]
called the \textit{amalgamation of $A$ with $B$ along $J$ with respect to $f$%
}. This construction is a generalization of the amalgamated duplication of a
ring along an ideal (cf., for instance \cite{PP, AMES, BMT, MMSN, DF, MCM}).

It is clear that a subset $S$ of $A$ is a multiplicatively closed subset of
$A$ if and only if $S\bowtie^{f}J:=\{(s,f(s)+j)\mid s\in S;j\in J\}$ is a
multiplicatively closed subset of $A\bowtie^{f}J.$

%We need the following lemma for the proof of the proposition.

\begin{lemma}
	\label{amal1}
	
	If $I$ is an ideal of $A,$ then
	\[
	\sqrt{I\bowtie^{f}J}=\sqrt{I}\bowtie^{f}J.
	\]
	
\end{lemma}

\begin{proof}
	
	Let $(a,f(a)+j)\in A\bowtie^{f}J$ such that $(a,f(a)+j)\in\sqrt{I\bowtie^{f}%
		J}.$ Then there exists a nonzero integer $n$ such that $(a,f(a)+j)^{n}%
	=(a^{n},f(a^{n})+j^{\prime})\in I\bowtie^{f}J$ where $j^{\prime}%
	=\displaystyle \sum_{k=0}^{n-1}\left(
	\begin{array}
	[c]{c}%
	n\\
	k\\
	\end{array}
	\right) f(a^{k})j^{n-k}\in J.$ Then $a^{n}\in I$, so $a\in\sqrt{I}$ and
	$(a,f(a)+j)\in\sqrt{I}\bowtie^{f}J$. Hence, we conclude that $\sqrt{I\bowtie^{f}J}%
	\subset\sqrt{I}\bowtie^{f}J.$ Conversely, let $(a,f(a)+j)\in\sqrt{I}%
	\bowtie^{f}J$. Then there exists $n\in\mathbb{N}^{*}$ such that $a^{n}\in I.$
	On the other hand $(a,f(a)+j)^{n}=(a^{n},f(a^{n})+j^{\prime }).$ Then
	$(a,f(a)+j)\in\sqrt{I\bowtie^{f}J}$, which yields $\sqrt{I}\bowtie
	^{f}J\subset\sqrt{I\bowtie^{f}J}.$ Finally, we conclude $\sqrt{I\bowtie^{f}J}=\sqrt
	{I}\bowtie^{f}J.$
\end{proof}

\begin{proposition}
Let $ S $ be a multiplicatively closed subset and $ I, J $ be ideals of $ A $. The following statements are equivalent:
	
	\begin{enumerate}
		\item $I$ is an $S$-1-absorbing primary ideal of $A.$
		
		\item $I\bowtie^{f}J$ is an $S\bowtie^{f}J$-1-absorbing ideal of $A\bowtie
		^{f}J.$
	\end{enumerate}
\end{proposition}

\begin{proof}
	
	Assume that $I\bowtie^{f}J$ is an $S\bowtie^{f}J$-1-absorbing primary ideal of
	$A\bowtie^{f}J$ and let $(s,f(s)+j)$ be an $S\bowtie^{f}J$-element of
	$I\bowtie^{f}J.$ Take $abc\in I$ for non-unit elements $a,b,c$ of $A.$ Then
	$(a,f(a)),(b,f(b)),(c,f(c))$ are also non-unit elements of $A\bowtie^{f}J$ and
	$(a,f(a))(b,f(b))(c,f(c))\in I\bowtie^{f}J.$ Since $I\bowtie^{f}J$ is an
	$S\bowtie^{f}J$-1-absorbing primary submodule, we get either
	$(s,f(s)+j)(a,f(a))(b,f(b))\in I\bowtie^{f}J$ or $(s,f(s)+j)(c,f(c))\in
	\sqrt{I\bowtie^{f}J}=\sqrt{I}\bowtie^{f}J$ by Lemma \ref{amal1}. Then
	either $sab\in I$ or $sc\in\sqrt{I}.$  Thus $I$ is an $S$-1-absorbing primary
	ideal of $A.$
	
	Conversely, suppose that $I$ is an $S$-1-absorbing primary ideal of $A$ and let
	$s\in S$ be an $S$-element of $I.$ Let $(a,f(a)+j),(b,f(b)+k),(c,f(c)+l)$ be
	non-unit elements of $A\bowtie^{f}J$ such that
	$(a,f(a)+j)(b,f(b)+k)(c,f(c)+l)\in I\bowtie^{f}J.$ Then $abc\in I.$It is necessary that one of the $a,b,c$ is non-unit, otherwise, $abc$ is unit and $abc\in I,$ a contradiction. Since $I$ is an $S$-1-absorbing primary submodule of $A,$ we get either
	$sab\in I$ or $sc\in\sqrt{I}$ which gives either $(s,f(s))(a,f(a)+j)(b,f(b)+k)\in
	I\bowtie^{f}J$ or $(s,f(s))(c,f(c)+l)\in\sqrt{I}\bowtie^{f}J=\sqrt
	{I\bowtie^{f}J}$ by Lemma \ref{amal1}. Hence $I\bowtie^{f}J$ is an
	$S\bowtie^{f}J$-1-absorbing primary ideal of $A\bowtie^{f}J.$
\end{proof}

\begin{corollary}
	Let $R$ be a ring and $I, K$ be ideals of $R$ with $K\cap S=\emptyset.$ Then
	the following statements are equivalent:
	
	\begin{enumerate}
		\item $K$ is an $S$-1-absorbing primary ideal of $R.$
		
		\item $K\bowtie I$ is an $S\bowtie I$-1-absorbing primary ideal of $R\bowtie I.$
	\end{enumerate}
\end{corollary}

\begin{proof}
	It is easy to see that $(K\bowtie I)\cap(S\bowtie I)=\emptyset.$ The rest of
	the proof follows from the previous proposition.
\end{proof}

Let $S_{2}$ be a multiplicatively closed subset of the ring $f(A)+J.$ Consider the set
\[
\overline{S_{2}}^{f}:=\{(r,f(r)+j)\mid j\in J, r\in A, f(r)+j\in S_{2}\}.
\]
It is clear that $\overline{S_{2}}^{f}$ is a multiplicatively closed subset of
$A\bowtie^{f}J.$

If $K$ is an ideal of $f(A)+J$, then the set
\[
\overline{K}^{f}:=\{(a,f(a)+j)\mid a\in A,j\in J,f(a)+j\in K\}
\]
is an ideal of $A\bowtie^{f}J.$

\begin{lemma}
	\label{amal2}
	
	Under the above notations, we have
	\[
	\sqrt{\overline{K}^{f}}=\overline{\sqrt{K}}^{f}.
	\]
	
\end{lemma}

\begin{proposition}
	Let $f:A\longrightarrow B$ be a ring homomorphism, $J, K$ be ideals of $B$ and
	$S_{2}$ be a multiplicatively closed subset of $f(A)+J$ with $K\cap
	S_{2}=\emptyset.$ Let
	\[
	\overline{K}^{f}=\{(a,f(a)+j)\mid a\in A, j\in J, f(a)+j\in K\}.
	\]
	Then the following statements are equivalent:
	
	\begin{enumerate}
		\item $K$ is an $S_{2}$-1-absorbing primary ideal of $f(A)+J.$
		
		\item $\overline{K}^{f}$ is an $\overline{S}^{f}_{2}$-1-absorbing primary ideal
		of $A\bowtie^{f}J.$
	\end{enumerate}
\end{proposition}

\begin{proof}
	
	Assume that $\overline{K}^{f}$ is an $\overline{S}^{f}$-1-absorbing primary
	ideal of $A\bowtie^{f}J$ and let $(s,f(s)+p)$ be an $\overline{S}^{f}$-element of
	$\overline{K}^{f}.$ Let $(f(a)+j)(f(b)+k)(f(c)+l)\in K$ for non-unit elements
	$f(a)+j,f(b)+k,f(c)+l\in f(A)+J.$ Then $(a,f(a)+j),(b,f(b)+k),(c,f(c)+l)$ are
	non-unit elements of $A\bowtie^{f}J$ and $(a,f(a)+j)(b,f(b)+k)(c,f(c)+l)
	\in\overline{K}^{f}.$ Thus, we have either $(s,f(s)+p)(a,f(a)+j)(b,f(b)+k)\in
	\overline{K}^{f}$ or $(s,f(s)+p)(c,f(c)+l)\in\sqrt{\overline{K}^{f}}.$ In the
	first case, we have $(f(s)+p)(f(a)+j)(f(b)+k)\in K.$ In the second case, we have
	$(f(s)+p)(f(c)+k)\in\sqrt{\overline{K}^{f}}=\overline{\sqrt{K}}^{f}$ by Lemma \ref{amal2} and hence
	$(f(s)+p)(f(c)+k)\in\sqrt{K}.$ We conclude that $K$ is an $S_{2}$-1-absorbing
	primary ideal of $B.$ Conversely, suppose that $K$ is an $S_{2}$-1-absorbing
	primary ideal of $B$ and let $f(s)+p$ an $S_{2}$-element of $K.$ Let
	$(a,f(a)+j)(b,f(b)+k)(c,f(c)+l)\in\overline{K}^{f}$ for non-unit elements
	$(a,f(a)+j),(b,f(b)+k),(c,f(c)+l)\in A\bowtie^{f}J.$ Then we get $(abc,
	(f(a)+j)(f(b)+k)(f(c)+l))\in\overline{K}^{f}.$ Now suppose that $f(a)+j, f(b)+k,
	f(c)+l$ are non-unit elements of $f(A)+J.$ Since $K$ is an $S_{2}$-1-absorbing primary
	ideal of $f(A)+J,$ we conclude that either $(f(s)+p)(f(a)+j)(f(b)+k)\in K$ or
	$(f(s)+p)(f(c)+l)\in\sqrt{K}.$ In the first case, we have
	$(s,f(s)+p)(a,f(a)+j)(b,f(b)+k)\in\overline{K}^{f}.$ In the second case, we
	have $(s,f(s)+p)(c,f(c)+l)\in\sqrt{\overline{K}^{f}}.$ Now, if one of $f(a)+j, f(b)+k, f(c)+l$ is unit, then the claim is clear. Hence
	$\overline{K}^{f}$ is an $\overline{S_{2}}^{f}$-1-absorbing primary ideal of
	$A\bowtie^{f}J.$
\end{proof}

%which is an $R\bowtie I$-module with the multiplication given by
%$$(r,r+i).(m,m^{\prime})=(rm,(r+i)m^{\prime})$$ where $e\in R, i\in I,$ and
%$(m,m^{\prime})\in M\bowtie I.$
%
%Let $N$ be a submodule of $M.$ It is clear that
%\[
%N\bowtie I:=\{(n,m)\in N\times M\mid n-m\in IM\}
%\]
%and
%\[
%\overline{N}:=\{(m,n)\in M\times N\mid n_{m}\in IM\}.
%\]
%are an $R\bowtie I$-submodule of $M\bowtie I.$ If $S$ is a multiplicatively
%closed subset of $A,$ then obviously $S\bowtie I:=\{(s,s+i)\mid s\in S,i\in
%I\}$ is a multiplicatively closed subset of $A\bowtie I$ and $\overline
%{S}:=\{(s,s+i)\mid s+i\in S,i\in I\}.$ are a multiplicatively closed subset
%of $A\bowtie I.$ {\color{red}If $S$ is a multiplicatively
%	closed subset of $A,$ then it is obvious that the sets $S\bowtie I:=\{(s,s+i)\mid s\in S,i\in
%	I\}$ and $\overline {S}:=\{(s,s+i)\mid s+i\in S,i\in I\}$ are multiplicatively closed subsets
%	of $A\bowtie I.$}

An $R$-module can be regarded as an $R\bowtie I$-module via the natural projection $R\bowtie I\longrightarrow R; (r,r+i)\longmapsto r,$ and an $R\bowtie I$-module can also be regarded as an $R$-module via the natural injection $\iota:R\longrightarrow
R\bowtie I; r\longmapsto(r,r).$

%%%%%%%%%%%%%%%%%%%%%%%%%%%%%%%%%%%%%%%%%%%%
%%%%%%%%%%%%%%%%%%%%%%%%%%%%%%%%%%%%%%%%%%%%%%%%%%%%%%%%%

Let $R$ be a ring, $I$ an ideal of $R$, $M$ an $R$-module and set
\[
M\bowtie I:=\{(m,m^{\prime})\in M\times M\mid m-m^{\prime}\in IM\},
\]
which is an $R\bowtie I$-module with the multiplication given by
$$ (r,r+i)(m,m^{\prime})=(rm,(r+i)m^{\prime}) $$
where $ r\in R $, $ i\in I $ and $ (m,m^{\prime})\in M\bowtie I. $
%\[
%(r,r+i)(m,m^{\prime})=(rm,(r+i)m^{\prime}),\quad\text{where}\quad r\in R, i\in
%I,\quad\text{and}\quad(m,m^{\prime})\in M\bowtie I.
%\]
Then the set $M\bowtie I$ is called the \textit{duplication of the $R$-module $M$ along the
	ideal $I.$}

Let $N$ be a submodule of $M.$ It is clear that
\[
N\bowtie I:\{(n,m)\in N\times M\mid n-m\in IM\}
\]
and
\[
\overline{N}:=\{(m,n)\in M\times N\mid n-m\in IM\}
\]
are submodules of $M\bowtie I.$ If $S$ is a multiplicatively closed subset of
$R,$ then it is obvious that the sets $S\bowtie I:=\{(s,s+i)\mid s\in S,i\in I\}$ and
$\overline{S}:=\{(s,s+i)\mid i\in I,s+i\in S\}$ are multiplicatively closed
subsets of $R\bowtie I.$ 
\begin{lemma}
	\label{lduplication1}
	
	Let $N$ be a submodule of $M.$ Then
	\[
	(N\bowtie I:_{R\bowtie I}M\bowtie I)=(N:_{R}M)\bowtie I.
	\]
	
\end{lemma}

\begin{proof}
	Let $(a,a+i)\in(N\bowtie I:_{R\bowtie I}M\bowtie I)$. Then $(a,a+i)(M\bowtie I)\subset N\bowtie I.$ Clearly,
	$aM\subset N$ and thus $(a,a+i)\in(N:_{R}M)\bowtie I.$ Conversely, suppose
	that $(a,a+i)\in(N:_{R}M)\bowtie I$ and take $(m,m^{\prime})\in M\bowtie I.$
	We have $(a,a+i)(m,m^{\prime})=(am,(a+i)m^{\prime})\in M\times M$ and
	$a(m-m^{\prime})-im^{\prime}\in IM$ since $m-m^{\prime}\in IM.$ Hence
	$(a,a+i)\in (N\bowtie I:_{R\bowtie I}M\bowtie I).$ Thus, we conclude that $(N\bowtie
	I:_{R\bowtie I}M\bowtie I)=(N:_{R}M)\bowtie I.$
\end{proof}

\begin{lemma}
	\label{lduplication2}
	
	Let $L$ be a submodule of $M.$ Then the following statements are equivalent:
	
	\begin{enumerate}
		\item $L$ is a prime submodule of $M.$
		
		\item $L\bowtie I$ is a prime submodule of $M\bowtie I.$
	\end{enumerate}
\end{lemma}

\begin{proof}
	Assume that $L\bowtie I$ is a prime submodule of $M\bowtie I.$ Let $am\in L$
	for $a\in R$ and $m\in M.$ Then $(a,a)(m,m)\in L\bowtie I.$ Since $L\bowtie I$
	is a prime submodule of $M\bowtie I,$ we get either $(a,a)\in(L\bowtie
	I:_{R\bowtie I}M\bowtie I)$ which implies $a\in (N:M)$ by Lemma
	\ref{lduplication1}, or $(m,m)\in L\bowtie I$ which follows $m\in L.$ Hence $L$
	is a prime submodule of $M.$ Conversely, assume that $L$ is a prime submodule of $ M $
	and let $(a,a+i)(m,m^{\prime})\in L\bowtie I$ for $(a,a+i)\in R\bowtie I$ and
	$(m,m^{\prime})\in M\bowtie I.$ Thus $(am,(a+i)m^{\prime})\in L\bowtie I$ and we
	conclude that $am\in L.$ Due to the fact that L is a prime submodule, we obtain
	either $a\in(N:M)$ or $m\in L.$ In the first case, we have $(a,a+i)\in
	(N:M)\bowtie I=(N\bowtie I:_{R\bowtie I}M\bowtie I).$ In the second case, we
	have $(m,m^{\prime})\in L\bowtie I,$ since $m-m^{\prime}\in IM.$ Thus,
	$L\bowtie I$ is a prime submodule of $M\bowtie I.$
\end{proof}

\begin{lemma}
	\label{lduplication3} Let $N$ be a submodule of $M.$ Then the prime submodule
	which contains $N\bowtie I$ are exactly of the form $L\bowtie I$ where $L$ is
	a prime submodule containing $N.$ Moreover, $ M\bowtie I$-$rad(N\bowtie I)=M$-$rad(N)\bowtie I. $

\end{lemma}

\begin{proof}
	Let $\varphi:M\bowtie I\longrightarrow M;(m,m^{\prime})\longmapsto m$ be an
	$R\bowtie I$-epimorphism with $\ker\varphi=0\times IM.$ Then $M\bowtie
	I/0\times IM\simeq M.$ So, if $L$ is a submodule of $M\bowtie I$ such that
	$0\times IM\subset I,$ then there exists a submodule $K$ of $M$ such that $L/0\times
	IM\simeq K,$ which implies $L=K\bowtie I.$ From Lemma \ref{lduplication2}, we
	conclude that $L$ is a prime submodule of $M\bowtie I$ if and only if $K$ is a
	prime submodule of $M.$ Now, if $L$ is a prime submodule of $M\bowtie I$ such
	that $N\bowtie I\subset L,$ then $0\times IM\subset L,$ and so there exits a
	prime submodule $K$ of $M$ with $L=K\bowtie I.$
	
	Now, it is clear that $\bigcap_i(K_i\bowtie
	I)=\left( \bigcap_iK_i\right) \bowtie I$ where $ K_i $'s are the submodules of $ M $ and
	consequently we have $M\bowtie I$-$rad(N\bowtie I)=M$-$rad(N)\bowtie I.$ 
\end{proof}

\begin{theorem}
	Let $N$ be a submodule of $M$ with $(N:_{R}M)\cap S=\emptyset.$ Then the
	following statements are equivalent:
	
	\begin{enumerate}
		\item $N$ is an $S$-1-absorbing primary submodule of $M.$
		
		\item $N\bowtie I$ is an $S\bowtie I$-1-absorbing primary submodule of $M\bowtie
		I.$
	\end{enumerate}
\end{theorem}

\begin{proof}
	From Lemma \ref{lduplication1}, we clearly note that $(N\bowtie
	I:_{R\bowtie I}M\bowtie I)\cap S\bowtie I=\emptyset$ if and only if
	$(N:_{R}M)\cap S=\emptyset.$ Now suppose that $N\bowtie I$ is an $S\bowtie I$-1-absorbing
	primary submodule of $M\bowtie I$. Let $(s,s+i)$ be an $S\bowtie I$-element of
	$N\bowtie I$ and let $abm\in N$ for non-unit elements $a,b\in R$ and $m\in M.$
	Then $(a,a)(b,b)(m,m)\in N\bowtie I.$ Since $(a,a),(b,b)$ are non-units and
	$N\bowtie I$ is an $S\bowtie I$-1-absorbing primary submodule of $M\bowtie I,$ we
	have either
	\[
	(s,s+i)(a,a)(b,b)\in(N\bowtie I:M\bowtie I)
	\]
	or%
	
	\[
	(s,s+i)(m,m)\in M\bowtie I-\text{rad}(N\bowtie I).
	\]
	In the first case, we have $sab\in (N:M)$ by Lemma \ref{lduplication1}. In the second case, $M\bowtie I$-$rad(N\bowtie I)=M$-$rad(N)\bowtie I$ by Lemma \ref{lduplication3} and we conclude that  $(sm,(s+i)m)\in M$-$rad(N)\bowtie I,$  and so $sm\in M$-$rad(N).$ If such an $ N $ does not exists, it is well known that $M$-$rad(N)=M.$ Then $N$ is an $S$-1-absorbing primary submodule of $M.$
	Conversely, assume that $N$ is an $S$-absorbing primary submodule of $M$. Let $s\in S$ be an $S$-element of $N$ and let $(a,a+i)(b,b+i)(m,m')\in N\bowtie I$ for non-unit elements $(a,a+i),(b,b+i)\in R\bowtie I$ and $(m,m')\in M\bowtie I.$ Then $abm\in N.$ If $a$ or $b$ is unit, the claim is clear. Now we can assume $a$ and $b$ are non-unit elements of $R.$ Due to the fact that $N$ is an $S$-1-absorbing primary submodule of $R,$ we get either $sab\in (N:M)$ or $sm\in M$-$rad(N).$ In the first case, we conclude that $(s,s)(a,a+i)(b,b+i)\in(N:M)\bowtie I=(N\bowtie I:M\bowtie I).$ In the second case, we have $(sm,sm'+im')\in M$-$rad(N)\bowtie I,$ since $r(m-m')-im'\in IM.$ Thus, the claim follows by $M$-$rad(N)\bowtie I=M\bowtie I$-$rad(N\bowtie I).$

\end{proof}

%\section*{Declarations}
%The authors did not receive support from any organization for the submitted work and have no relevant financial or non-financial interests to disclose.

\end{document}